\documentclass{article}
\usepackage{amsfonts, amssymb, amsthm,graphicx,amsmath,amsfonts,color}

\pdfoutput=1
\DeclareMathOperator{\grad}{grad}

\theoremstyle{plain}
\newtheorem{teorema}{Theorem}[section]
\newtheorem{proposicion}[teorema]{Proposition}
\newtheorem{lema}[teorema]{Lemma}
\newtheorem{corolario}[teorema]{Corollary}

\theoremstyle{definition}
\newtheorem{definicion}[teorema]{Definition}
\newtheorem{remark}[teorema]{Remark}

\begin{document}

\title{Null screen isoparametric hypersurfaces in Lorentzian space forms}

\author{Matias Navarro\footnote{Corresponding author. Facultad de Matem\'aticas, Universidad Aut\'onoma de Yucat\'an, Perif\'erico Norte, Tablaje 13615, M\'erida, Yucat\'an, MEXICO. Partially supported by CONACYT grant 457490 under Project FMAT-2016-0013 of UADY. matias.navarro@correo.uady.mx} \and Oscar Palmas\footnote{Departamento de Matem\'aticas, Facultad de Ciencias, UNAM. CP 04510, Ciudad de M\'exico, MEXICO. Partially supported by UNAM under Project PAPIIT-DGAPA IN113516. oscar.palmas@ciencias.unam.mx} \and Didier Solis\footnote{Facultad de Matem\'aticas, Universidad Aut\'onoma de Yucat\'an, Perif\'erico Norte, Tablaje 13615, M\'erida, Yucat\'an, MEXICO. Partially supported by UADY under Project PROFOCIE 2015-12-1918. didier.solis@correo.uady.mx} }

%\author{}
%
%
%\author[uady]{Matias Navarro\corref{cor1}\fnref{fn1}}
%\ead{matias.navarro@correo.uady.mx}
%
%\author[fcunam]{Oscar Palmas\fnref{fn2}}
%\ead{oscar.palmas@ciencias.unam.mx}
%
%\author[uady]{Didier A. Solis\fnref{fn3}}
%\ead{didier.solis@correo.uady.mx}
%
%\cortext[cor1]{Corresponding author.}
%
%
%\fntext[fn1]{Partially supported by CONACYT grant 457490 under Project FMAT-2016-0013 of UADY.}
%
%\fntext[fn2]{Partially supported by UNAM under Project PAPIIT-DGAPA IN113516.}
%
%\fntext[fn3]{Partially supported by UADY under Project PROFOCIE 2015-12-1918.}
%
%\address[uady]{Facultad de Matem\'aticas, Universidad Aut\'onoma de Yucat\'an, Perif\'erico Norte, Tablaje 13615, M\'erida, Yucat\'an, MEXICO.}
%\address[fcunam]{Departamento de Matem\'aticas, Facultad de Ciencias, UNAM. CP 04510, Ciudad de M\'exico, MEXICO.}

\maketitle

\begin{abstract}
{In this paper we develop the notion of screen isoparametric hypersurface for null hypersurfaces of Robertson-Walker spacetimes. Using this formalism we derive Cartan identities for the screen principal curvatures of null screen hypersurfaces in Lorentzian space forms and provide a local characterization of such hypersurfaces.}
\end{abstract}

\section{Introduction}

The study of submanifold geometry is as old as the field of differential geometry itself, dating back to the work of Gauss. Among the most important results, we find the classification of submanifolds that satisfy certain geometric restrictions, often times related to the appearance of high degrees of symmetry. In the context of Lorentzian Geometry and General Relativity, there is a special kind of submanifolds in which the classical theory falls apart, namely, those in which the restriction of the ambient spacetime metric is degenerate. In particular, many of the most remarkable features of Relativity are expressed geometrically in terms of degenerate or null hypersurfaces: event horizons, achronal boundaries, conformal infinity for asymptotically flat spacetimes, etc. \cite{MR1172768, MR0424186, MR757180} In recent years, a framework designed to study null submanifolds in the semi-Riemannian setting that closely resembles its Riemannian counterpart was developed by Duggal et al. \cite{MR1383318, MR2598375}. Following this seminal work, numerous concepts and results pertaining the classical theory have found analogs in the degenerate case. One of the most important class of solutions to the Einstein equations --present since the dawn of Relativity-- is the family of Generalized Robertson-Walker spacetimes, which model an isotropic and homogeneous universe  \cite{MR1172768, MR0424186, MR757180}. This family encompasses warped products with constant curvature fibers, including open portions of all Lorentzian space forms. In recent years, there has been a growing interest in the submanifold geometry of Generalized Robertson-Walker spacetimes \cite{MR3187030, MR3397377, MR2978423, MR3508919}.

In the Riemannian case, isoparametric hypersurfaces can be described in several different ways. Originally, isoparametric hypersurfaces arose in the study of wave propagation. In order to characterize wavefronts  that move at a constant speed, Laura  \cite{Laura} established the conditions $|\nabla F|^2=\varphi_1(F)$ and $\Delta F=\varphi_2 (F)$, where $F\colon \mathbb{R}^3\to\mathbb{R}$ is a smooth function describing the wave and $\varphi_i\colon\mathbb{R}\to\mathbb{R}$ are two smooth functions. The wavefronts, that is, the level sets $F^{-1}(k)$ of regular values of $F$ were dubbed isoparametric. As a consequence, isoparametric hypersurfaces have constant mean curvature and constant principal curvatures \cite{Somigliana, LeviCivita}. This latter property has become the standard geometrical definition of an \emph{isoparametric hypersurface}. The program of classifying  isoparametric hypersurfaces in Riemannian manifolds of constant curvature $c$ dates back to the work of Somigliana, Levi-Civita, Segre and Cartan \cite{Somigliana, LeviCivita, Segre, MR1553310, MR0000169, MR0004519} mainly in the late 1930's and it is still not complete to this day. The main breakthrough was due to Cartan who showed that for the $l$ distinct constant principal curvatures $\lambda_j$ of multiplicity $m_j$, 
\[
\sum_{j\neq i}m_j\frac{1+\lambda_i\lambda_j}{\lambda_i -\lambda_j}=0 ,\quad 1\le j\le l,
\] 
which are known thereafter as \emph{Cartan identities}. Using these relations, a complete classification for Riemannian space forms of nonpositive curvature was established \cite{MR1553310, MR0000169, MR0004519}. For $c=0$, any isoparametric hypersurface $M^n$ immersed in $\mathbb{R}^{n+1}$ has at most two distinct principal curvatures and it is an open portion of a hyperplane, a sphere or a cylinder over a sphere. For $c<0$, an isoparametric hypersurface $M^n\subset\mathbb{H}^{n+1}$ has also at most two distinct principal curvatures and it is an open portion of a a totally geodesic hypersurface, a horosphere, an equidistant surface or a tube over a totally geodesic submanifold. The positive curvature case has proven far more difficult to deal with. Cartan classified the cases with at most three distinct principal curvatures. The next crucial step came in the 1980's due to the work of M\"{u}nzner, \cite{Munzner1, Munzner2} who showed that the number $l$ of distinct principal curvatures must be either $1$, $2$, $3$, $4$ or $6$. The cases $n=4$ and $n=6$ have not been fully classified yet and, in fact, it is one of the problems in the influential S. T. Yau's list of outstanding problems in geometry \cite{Yau,CecilChiJensen,ChiIII,Miyaoka}.

The notion of isoparametric hypersurface has been generalized to different geometric scenarios, like complex manifolds \cite{Takagi,Montiel} and semi-Riemannian space forms \cite{MR753432}, to name a few. In the context of Lorentzian geometry, Nomizu established the first classification results for spacelike isoparametric hypersurfaces in Lorentzian space forms \cite{MR728336}. This classification continues to be refined even to this day \cite{LiXie,SMartin}. On the other hand, timelike (i.e. Lorentzian) isoparametric hypersurfaces were first studied by Magid in Lorentz-Minkowski spacetimes \cite{MR783023} and later on by Li and Wang \cite{MR2202686} and Xiao \cite{MR1696128} in de Sitter and anti de Sitter spacetimes, respectively. Notice that in this case the shape operator is not diagonalizable due to the Lorentzian character of the induced metric. Hence the notion of isoparametric hypersurfaces was broaden by requiring that the minimal polynomial of the shape operator has constant coefficients. The classification of such hypersurfaces is not complete either. Finally, when it comes to null hypersurfaces, Atindogbe et al \cite{MR3270005} have provided a first local result, namely, that a null screen conformal hypersurface in a Lorentzian space form must be locally a null triple product. 

In this work we present a general framework to deal with null isoparametric hypersurfaces in Robertson-Walker spacetimes and provide finer results than those reported in \cite{MR3270005}.  This article is organized as follows: in section 2 we introduce the preliminary results pertaining null submanifold theory, following closely the approach on \cite{MR1383318, MR2598375}. Then on section 3 we use the transnormal approach developed in \cite{MR3508919} to establish our framework for Generalized Robertson-Walker spacetimes. Finally, on section 4 we establish Cartan identities for null isoparametric hypersurfaces immersed in a Lorentzian space form and derive a local characterization result.

\section{Preliminaries}

We follow closely the notation and results in \cite{MR1383318} and \cite{MR2598375}. Let $(\bar M^{n+2},\bar g)$ be a $(n+2)$-dimensional, semi-Riemannian manifold with metric $\bar g$ of constant index $q\in\{1,\dots,n+1\}$. A hypersurface $M$ of $\bar M$ is {\em null} if the {\em radical bundle} $\mathrm{Rad}(TM)=TM\cap TM^\perp$ is different from zero at each $p\in M$; or equivalently, if the restriction $g$ of $\bar g$ to $M$ has rank $n$ everywhere.

A {\em screen distribution} $S(TM)$ on $M$ is a non-degenerate vector bundle complementary to $TM^\perp$. A null hypersurface with an specific screen distribution is denoted $(M,g,S(TM))$. From \cite{MR1383318}, we know that there is a vector bundle $\mathrm{tr}(TM)$ of rank $1$ over $M$, called the {\em transversal bundle}, such that for each non-zero section $\xi\in\Gamma(TM^\perp)$ defined in an open set $U\subset M$ there is a unique section $N\in\Gamma(\mathrm{tr}(TM))$ such that
\begin{equation}\label{eq:propiedadesdeN}
\bar g(\xi,  N )=1, \quad
\bar g( N , N )=\bar g( N ,X)=0
\end{equation}
for each $X\in\Gamma(S(TM\vert_{U}))$. We will work hereafter in a maximal neighborhood $U$ with these properties and omit the reference to it. We write
\begin{equation}  \label{eq:descomposicion0}
T\bar M\vert_M=TM\oplus \mathrm{tr}(TM).
\end{equation}
and\begin{equation}  \label{eq:descomposicion1}
TM=S(TM)\oplus_{\mathrm{orth}} \mathrm{Rad}(TM),
\end{equation}
so that
\begin{equation*}
T\bar M\vert_M=S(TM)\oplus_{\mathrm{orth}}(\mathrm{Rad}(TM)\oplus\mathrm{tr}(TM)).
\end{equation*}

Let us denote by $P$ the projection of $\Gamma(TM)$ onto $\Gamma(S(TM))$ using the decomposition (\ref{eq:descomposicion1}).

Now let $\bar\nabla$ be the Levi-Civita connection of $\bar M$. The \emph{local Gauss-Weingarten formulae} relative to the above decompositions are
\begin{equation}  \label{eq:gauss1}
\begin{array}{rcl}
\bar\nabla_{X}Y & = &  \nabla_{X}Y+h(X,Y) = \nabla_{X}Y+B(X,Y) N, \\[0.1cm]
\bar\nabla_{X} N & = & -A_{ N }X+\nabla_X^tN = -A_{ N }X+\tau(X)  N; \\[0.1cm]
\nabla_{X}PY & = & \nabla_{X}^*PY+h^*(X,PY) = \nabla_{X}^*PY+C(X,PY)\xi; \\[0.1cm]
\nabla_{X} \xi & = & -A_{\xi}^*X+\nabla_X^{*t} \xi= -A_{\xi}^*X-\tau(X) \xi, 
\end{array}
\end{equation}
where $X,Y\in\Gamma(TM)$. Here $\nabla$, $\nabla^t$, $\nabla^*$ and $\nabla^{*t}$ denote the induced connections on $TM$, $\mathrm{tr}(TM)$, $S(TM)$ and $\mathrm{Rad}(TM)$, respectively; $h$ and $h^*$ are the \emph{second fundamental forms} of $M$ and $S(TM)$, while
\begin{equation}\label{eq:2ffB}
B(X,Y) = \bar g( \bar\nabla_{X}Y, \xi ) = \bar g(h(X,Y),\xi)=g( A_\xi^*X,Y),
\end{equation}
\begin{equation}\label{eq:2ffC}
C(X,PY) = \bar g( \nabla_X PY, N)= \bar g(h^*(X,PY),\xi)=g( A_NX,PY),
\end{equation}
are the \emph{local second fundamental forms} of $M$ and $S(TM)$, $A_ N $ and $A_\xi^*$ are the {\em local shape operators} of $M$ and $S(TM)$ and $\tau$ is the $1$-form on $\Gamma(TM)$ given by
\begin{equation}\label{eq:tau}
\tau(X)=\bar g(\bar\nabla_X N ,\xi)=\bar g(\bar\nabla_X^t N ,\xi).
\end{equation}

\begin{proposicion}\label{prop:propiedades}
The following properties hold true:
\begin{enumerate}
\item $\bar g(A_NX,N)=0$ for every $X\in\Gamma(TM)$;
\item $A_\xi^*\xi=0$;
\item $A_\xi^*$ is symmetric relative to $g$, that is,
\[
g(A_ \xi^*X,Y)=g(X,  A_\xi^*Y)
\]
for each $X,Y\in\Gamma(TM)$.
\end{enumerate}
\end{proposicion}

For the proofs, see the references \cite{MR1383318}, \cite{MR2598375} and \cite{MR719023}.

\section{Null hypersurfaces in GRW spacetimes}\label{sec:nullinwp}

In later sections we will study null hypersurfaces in Lorentzian space forms, but in order to motivate some of our definitions and results, here we consider the more general case when the ambient space $\bar M$ is a {\em Generalized Robertson-Walker (GRW) spacetime}; that is, a Lorentzian warped product of the form $-I\times_{\varrho} F$, where $F$ is a $(n+1)$-dimensional Riemannian manifold and $\varrho$ is a differentiable, positive function defined in a real interval $I\subset\mathbb R$. We recall also that a {\em classical Robertson-Walker spacetime} is a GRW spacetime where the fiber $F$ has constant curvature. In fact, (see \cite{MR1464181} or \cite[p. 267]{MR2371700}) we know that $\bar M$ has constant sectional curvature $\bar c$ if and only if $F$ has constant sectional curvature $c$ and
\begin{equation}\label{eq:ARS}
\frac{\varrho''}{\varrho} = \bar c = \frac{c+(\varrho')^2}{\varrho^2}.
\end{equation}

The standard examples are the Lorentz-Minkowski,  de Sitter and anti de Sitter spacetimes, whose definitions we recall for completeness. If $\mathbb R_k^m$ denotes the $m$-dimensional real vector space with metric of index $k$, then the {\em $(n+2)$-dimensional Lorentz-Minkowski spacetime} is $\mathbb R_1^{n+2}$; the {\em $(n+2)$-dimensional de Sitter space} is defined by
\[
\mathbb S_1^{n+2}=\{\,p\in\mathbb R_1^{n+3}\,\vert\,\bar g( p,p) =1\,\}
\]
while the {\em $(n+2)$-dimensional anti de Sitter space} is
\[
\mathbb H_1^{n+2}=\{\,p\in\mathbb R_2^{n+3}\,\vert\,\bar g( p,p) =-1\,\}.
\]

Each of these spaces may be modelled also as a warped product; clearly,
\[
\mathbb R_1^{n+2}=-\mathbb R\times\mathbb R^{n+1};
\]
while the de Sitter space is isometric to $-\mathbb R\times_{\cosh }\mathbb S^{n+1}$. In fact, the transformation
\begin{equation}\label{eq:isometria1}
-\mathbb R\times_{\cosh }\mathbb S^{n+1}\to\mathbb S_1^{n+2},\qquad (t,p)\mapsto(\sinh t,\cosh t\, p)
\end{equation}
is an isometry. Similarly, the isometry 
\begin{equation}\label{eq:isometria2}
-\mathbb R\times_{\cos }\mathbb H^{n+1}\to\mathbb H_1^{n+2},\qquad (t,p)\mapsto(\sin t,\cos t\, p).
\end{equation}
maps the warped product onto an unbounded open region of the anti de Sitter spacetime.

\medskip

Now, let us return to the general case of a GRW spacetime. We recall briefly some results from \cite{MR3508919}. Let $F$ be a Riemannian manifold and $f:F\to\mathbb R$ be a differentiable function. Then the graph of $f$ given as
\[
\{\,(f(p),p)\,\vert\,p\in F\,\}
\]
is a null hypersurface in $-I\times_{\varrho} F$ if and only if $f$ is a {\em transnormal function} satisfying
\begin{equation}\label{eq:transnormal}
\vert \grad f\vert =\varrho\circ f;
\end{equation}
we write $\varrho\circ f$ simply as $\varrho$. In this context, set
\[
\xi=\frac{1}{\sqrt{2}}\left(1,\frac{\grad f}{\varrho^2}\right).
\]

We choose the screen distribution $S^*(TM)$ as the family of tangent bundles of the level hypersurfaces $S_t=M\cap\left(\{t\}\times F\right)$ and
\[
N=\frac{1}{\sqrt{2}}\left(-1,\frac{\grad f}{\varrho^2}\right).
\]

\begin{lema}\label{lema:tau} $\tau(X)=0$ for any $X\in\Gamma(S^*(TM))$; in consequence, 
\begin{equation}\label{eq:transversales}
\nabla_X^tN=0,\quad\quad \nabla_X^{*t}\xi=0
\end{equation}
or equivalently, $\bar\nabla_XN=-A_NX$ and $\bar\nabla_X\xi=-A_\xi^*X$. Moreover, $A_N\xi=0$.
\end{lema}

\begin{proof} Although the authors proved these facts in \cite{MR3508919}, we give the proof of the key fact $\tau(X)=0$ for completeness. If $\eta=(0,\grad f/\varrho^2)$ and $t$ is the standard coordinate in the real interval $I$, then
\[
\tau(X)=\bar g(\bar\nabla_XN,\xi)=\frac{1}{2}\bar g(\bar\nabla_X(\eta-\partial_t),\eta+\partial_t).
\]

Note that $\eta$ is a unit spacelike vector field, while $\partial_t$ is a unit timelike vector field, implying $\bar g(\bar\nabla_X\eta,\eta)=0$ and $\bar g(\bar\nabla_X\partial_t,\partial_t)=0$. On the other hand, $\bar g(\eta,\partial_t)=0$ and hence
\[
\bar g(\bar\nabla_X\eta,\partial_t)=-\bar g(\eta,\bar\nabla_X\partial_t)=-\bar g\left( \eta,\frac{\varrho'}{\varrho}X\right)=-\frac{\varrho'}{\varrho}X(f)=0,
\]
where in the second equality we used the formulae for the connection in warped products (see \cite[p. 206]{MR719023}) and the last equality is a consequence of the fact that $X$ is tangent to a level hypersurface of $f$.
\end{proof}

Based on these facts, the authors proved in \cite[Prop. 5.2]{MR3508919} the following relation between the shape operators (notice that a minus sign is missing in equation (19) there): For any $X\in\Gamma(S^*(TM))$,
\[
\frac{1}{\sqrt{2}}(A_N-A_\xi^*)X=\frac{\varrho'}{\varrho}X.
\]

Since $A_\xi^*\xi=0$ and $A_N\xi=0$, we thus have
\begin{equation}\label{eq:shape}
\frac{1}{\sqrt{2}}(A_N-A_\xi^*)=\frac{\varrho'}{\varrho}P,
\end{equation}
where $P:\Gamma(TM)\to \Gamma(S^*(TM))$ is the natural projection.

\begin{remark}
It is worth noting that the above relation resembles that of a {\em screen conformal} hypersurface (see \cite{MR1383318, MR2598375} and references therein). Recall that $(M,g,S(TM))$ is locally screen conformal iff $A_N=\varphi A_\xi^*$ for some function $\varphi$. This notion has proved useful for studying the geometry of $M$ through the geometry of its screen distribution. A key fact for many calculations in the screen conformal context is that $\tau(X)=0$ for any $X\in\Gamma(TM)$. In our warped product context we have this condition for $X\in\Gamma(S^*(TM))$, while  for $\xi$ we have
\begin{eqnarray*}
\tau(\xi) & = & \bar g(\bar\nabla_\xi N,\xi) =\frac{1}{2}\bar g(\bar\nabla_\xi(\eta-\partial_t),\eta+\partial_t) \\
   & = & -\frac{1}{\sqrt{2}}\bar g(\bar\nabla_{(\eta+\partial_t)}\partial_t,\eta) = -\frac{1}{\sqrt{2}}\frac{\varrho'}{\varrho}\bar g(\eta,\eta)
=  -\frac{1}{\sqrt{2}}\frac{\varrho'}{\varrho},
\end{eqnarray*}
so that $\tau(\xi)=0$ if and only if $\varrho'=0$.
\end{remark}

The above remark motivates the definition of a {\em screen locally quasi-conformal} null hypersurface $(M,g,S(TM))$ of a semi-Riemannian manifold as a null hypersurface whose shape operators satisfy
\[
A_N=\varphi A_\xi^*+ \psi P,
\]
in $\Gamma(TM)$ for some smooth functions $\varphi$ and $\psi$. This class of hypersurfaces will be studied in a forthcoming paper \cite{nosotros2}.

\section{Null screen isoparametric hypersurfaces}

Let $(\bar M_{\bar c}^{n+2},\bar g)$ be a Lorentzian manifold of constant curvature $\bar c$ and let $(M,g,S(TM))$ be a null hypersurface of $\bar M$. Since the shape operator $A_\xi^*$ is diagonalizable and $A_\xi^*\xi=0$, we have a frame field $\{\xi,E_1,\dots,E_n\}$ of eigenvectors of $A_\xi^*$ such that $\{E_1,\dots,E_n\}$ is an orthonormal frame field of $S(TM)$. Following \cite{MR3270005}, if $A_\xi^*E_i=\lambda_i E_i$, $i=1,\dots,n$, we call $\lambda_i$ the {\em screen principal curvatures} of $(M,g,S(TM))$.

\begin{definicion} Let $(\bar M_{\bar c}^{n+2},\bar g)$ be a Lorentzian manifold of constant curvature $\bar c$ and $(M,g,S(TM))$ a null hypersurface of $\bar M$. $(M,g,S(TM))$ is a {\em screen isoparametric null hypersurface} if all the screen principal curvatures are constant along $S(TM)$. 
\end{definicion}

Let $(\bar M_{\bar c}^{n+2},\bar g)$ be a Lorentzian manifold of constant curvature $\bar c$. We may multiply the metric $\bar g$ by a suitable constant to reduce to the cases $\bar c=-1,0,1$. On the other hand, since we will work locally, we may also suppose that $\bar M$ is an open subset of a warped product $-I\times_\varrho F$, where
\begin{equation*}
\begin{array}{lll}
\varrho(t)=\cos t, & F=\mathbb H^{n+1}, & \mbox{ for }\bar c=-1; \\
\varrho(t)=1, & F=\mathbb R^{n+1}, & \mbox{ for }\bar c=0; \\
\varrho(t)=\cosh t, & F=\mathbb S^{n+1}, & \mbox{ for }\bar c=1.
\end{array}
\end{equation*}

If $M$ is a null hypersurface of $\bar M_{\bar c}$, then $M$ is the graph of a transnormal function $f$ satisfying $\vert\grad f\vert=\varrho\circ f$; finally, we choose the screen distribution $S^*(TM)$ as in Section \ref{sec:nullinwp}, namely, as the familiy of tangent bundles of the level hypersurfaces of $f$.

As noted in \cite{MR1383318, MR2598375} the Ricci tensor derived from $\nabla$ is not symmetric in general. Since the Ricci tensor is an important feature in the context of General Relativity (for instance, via the Raychaudhuri equation or various energy conditions \cite{MR1172768, MR0424186, MR757180}. Thus the following result suggests that this setting is the right one in order to study null hypersurfaces in GRW spacetimes, from both the geometric and physical point of view.

\begin{proposicion}\label{lema:dtau}
Let $(M,g,S^*(TM))$ be the null hypersurface of the Lorentzian warped product $-I\times_\varrho F$ given as the graph of a transnormal function $f$. Then $d\tau=0$ and the Ricci tensor of $\nabla$ is symmetric.
\end{proposicion}

\begin{proof}
Since $S^*(TM)$ is a family of tangent bundles, it is an integrable distribution, so that If $X,Y\in\Gamma(S(TM))$, then $[X,Y]\in\Gamma(S^*(TM))$. Recall (Lemma \ref{lema:tau}) that the $1$-form $\tau$ vanishes along $S^*(TM)$. Therefore,
\[
2d\tau(X,Y)=X(\tau(Y))-Y(\tau(X))-\tau([X,Y]=0.
\]
The other interesting case is when $X\in\Gamma(S(TM))$ and $Y=\xi$:
\begin{eqnarray*}
2d\tau(X,\xi) & = & X(\tau(\xi))-\xi(\tau(X))-\tau([X,\xi]) \\
 & = & -X\left(\frac{1}{\sqrt{2}}\frac{\varrho'}{\varrho}\right)-\tau(\nabla_X\xi-\nabla_\xi X) \\
 & = & -\tau(-A_\xi^*X-\tau(X)\xi-\nabla_\xi^*X-C(\xi,X)\xi) =C(\xi,X)\tau(\xi),
\end{eqnarray*}
but by (\ref{eq:shape}),
\[
C(\xi,X)=B(\xi,X)+\sqrt{2}\,\frac{\varrho'}{\varrho}g(\xi,X)=0.
\]

The assertion on the Ricci tensor is consequence of Theorem 2.4.1 in \cite{MR2598375}.
\end{proof}

Suppose further that $(M,g,S^*(TM))$ is a null screen isoparametric hypersurface of $\bar M_{\bar c}=-I\times_\varrho F$. One may ask if each level set  $S_t=M\cap\left(\{t\}\times F\right)$ is an isoparametric codimension $2$ submanifold of $\bar M_{\bar c}$. Since $\tau(X)=0$ for any $X\in\Gamma(S^*(TM))$, we have
\[
\nabla_X^\perp N=\nabla_X^\perp\xi=0
\]
for any $X\in\Gamma(TS_t)$; here $\nabla^\perp$ is the normal connection of $S_t$ as a non-degenerate submanifold of $\bar M$. From this it is easy to conclude that the normal curvature $R^\perp$ of $S_t$ satisfies $R_{XY}^\perp N=R_{XY}^\perp \xi=0$, i.e, the normal bundle of $S_t$ in $\bar M$ is flat. Thus we have the first condition for $S_t$ to be isoparametric in $\bar{M}_{\bar{c}}$ (see \cite{MR3408101}).

Let us now investigate under which conditions the principal curvatures of the shape operator $A_\eta$ of a parallel section $\eta$ of the unit normal bundle of $S_t$ are constant. If $\eta=a\xi+bN$ is a section of this bundle,
\[
\nabla_X^\perp\eta=\nabla_X^\perp(a\xi+bN)=X(a)\xi+X(b)N,
\]
and the condition for $\eta$ to be parallel is $X(a)=X(b)$, i.e., $a$ and $b$ are constant along $S_t$. From (\ref{eq:shape}) we have
\begin{equation}\label{eq.18prima}
A_\eta X=aA_\xi^*X+bA_NX=\left( a+b\right)A_\xi^*X+b\sqrt{2}\,\frac{\varrho'}{\varrho}X.
\end{equation}

Using that $\varrho$ is constant along $S_t$, we have:

\begin{proposicion}\label{prop:iso}
Let $(M,g,S^*(TM))$ be a null hypersurface of the Lorentzian warped product $\bar M_{\bar c}= -I\times_\varrho F$ with constant curvature $\bar c=-1,0,1$, $M$ given as the graph of a transnormal function $f$. $M$ is screen isoparametric if and only if  each slice $S_t$ is an isoparametric submanifold of $\bar M_{\bar c}$. Moreover, $S_t$ is isoparametric in the slice $\{t\}\times F$.
\end{proposicion}

\begin{proof}
The first claim was proved before the statement of the proposition, so it remains to prove the last assertion. Recalling Section \ref{sec:nullinwp}, since $S_t$ is nothing but a level set of the function $f$, the vector field
\[
\eta=\left(0,\frac{\grad f}{\varrho^2} \right)=\frac{1}{\sqrt{2}}(\xi+N)
\]
is the unit normal to $S_t$ in the slice $\{t\}\times F$. From this and (\ref{eq.18prima}),
\[
A_\eta X=\sqrt{2}\,A_\xi^*X + \frac{\varrho'}{\varrho}X, \qquad X\in\Gamma(TS_t),
\]
so that the curvatures of $S_t$ have the form
\[
\nu_i=\sqrt{2}\,\lambda_i+\frac{\varrho'}{\varrho},
\]
which are constant along $S_t$, and so the slice is isoparametric in $\{t\}\times F$.
\end{proof}

We use Proposition \ref{prop:iso} to derive Cartan identities for the screen principal curvatures of a null screen isoparametric hypersurface.

\begin{teorema}\label{teo:cartan}
Let $(M,g,S^*(TM))$ be a null screen isoparametric hypersurface of the Lo\-rentzian warped product $\bar M_{\bar c}= -I\times_\varrho F$ with constant curvature $\bar c=-1,0,1$, $M$ given as the graph of a transnormal function $f$. Let $\lambda_1,\dots,\lambda_l$ be the distinct screen principal curvatures of $M$ with multiplicities $m_1,\dots,m_l$. If $l>1$, for each $i=1,\dots,l$ we have
\begin{equation}\label{eq:cartan}
\sum_{j\ne i} m_j\frac{\bar c+2\lambda_i\lambda_j+\psi(\lambda_i+\lambda_j)}{\lambda_i-\lambda_j}=0,\qquad\mathrm{where}\ \psi=\sqrt{2}\,\frac{\varrho'}{\varrho}.
\end{equation}
\end{teorema}

\begin{proof}
We use the notation of Section \ref{sec:nullinwp} and that of Proposition \ref{prop:iso}. Recall that  the principal curvatures of $S_t$ are
\[
\nu_i=\frac{1}{\sqrt{2}}(2\lambda_i+\psi),\quad \psi=\sqrt{2}\,\frac{\varrho'}{\varrho}.
\]

Note that the metric in the slice $\{t\}\times F$ is that of $F$ multiplied by the factor $\varrho^2$, so if $c$ is the curvature of $F$, then $c/\varrho^2$ is the curvature of this slice. Applying the usual Cartan's identities to the isoparametric hypersurface $S_t$ (see \cite[p. 91]{MR3408101}), we have
\begin{eqnarray*}
0 & = & \sum_{j\ne i} m_j\frac{\dfrac{c}{\varrho^2}+\nu_i \nu_j}{\nu_i-\nu_j} 
= \sum_{j\ne i} m_j\frac{\dfrac{c}{\varrho^2}+\dfrac{1}{2}(2\lambda_i+\psi)(2 \lambda_j+\psi)}{\sqrt{2}(\lambda_i-\lambda_j)} \\
 & = & \sum_{j\ne i} m_j\frac{\dfrac{c}{\varrho^2}+2\lambda_i \lambda_j+\psi(\lambda_i+\lambda_j)+\dfrac{1}{2}\psi^2}{\sqrt{2}(\lambda_i-\lambda_j)} \\
 & = & \sum_{j\ne i} m_j\frac{\dfrac{c+(\varrho')^2}{\varrho^2}+2\lambda_i \lambda_j+\psi(\lambda_i+\lambda_j)}{\sqrt{2}(\lambda_i-\lambda_j)};
\end{eqnarray*}
to finish the proof we recall equation (\ref{eq:ARS}) relating $c$ and $\bar c$.
\end{proof}

\begin{corolario}
Let $(M,g,S^*(TM))$ be a null screen isoparametric hypersurface of $\bar M_{\bar c}$, $\bar c=0,-1$. Then the number $l$ of distinct screen principal curvatures of $M$ is at most $2$. If $l=2$ and $\bar c=0$, one of the screen principal curvatures is zero.
\end{corolario}

\begin{proof} {Although the Corollary follows from Theorem \ref{teo:cartan}, for completeness we give here an argument similar in spirit to that used in the Riemannian case.} Suppose $\bar c=0$ and $l\ge 2$. Suppose one of the screen principal curvatures is positive and let $\lambda_1$ be the smallest positive one; (\ref{eq:cartan}) becomes
\[
\sum_{j\ne 1} m_j\frac{\lambda_1\lambda_j}{\lambda_1-\lambda_j}=0;
\]
now every term in the above sum is non-positive, hence equal to zero. From this it follows that $l\le 2$, hence $l=2$, and one of the distinct screen principal curvatures is zero.

Now suppose that $\bar c=-1$, $l\ge 2$ and that at least one of the principal curvatures $\nu_i$ is positive. Moreover, suppose that $\nu_1$ is a positive principal curvature such that for any other principal curvature $\nu_j$, the number $\varrho\nu_j$ does not lie between $1/(\varrho\nu_1)$ and $\varrho\nu_1$. This implies that each term in the sum.
\[
\sum_{j\ne 1} m_j\frac{\nu_1 \nu_j - \dfrac{1}{\varrho^2}}{\nu_1-\nu_j} =0
\]
is non-positive, hence equal to zero. This in turn implies that $\varrho^2\nu_1\nu_i=1$ and therefore there are only two distinct screen principal curvatures.
\end{proof}

\begin{remark} {The classification results of null isoparametric hypersurfaces in $\mathbb R_1^{n+2}$ given in \cite{MR3270005} assume that the number $l$ of distinct screen principal curvatures must be at most $2$, but as we have seen this assumption on $l$ is not needed.}
\end{remark}

In view of our results, we may give a (local) characterization theorem for null isoparametric hypersurfaces in Lorentzian space forms. Let $(M,g,S^*(TM))$ be the null hypersurface of the Lorentzian warped product $-I\times_\varrho F$ with constant curvature $\bar c$, $M$ given as the graph of a transnormal function $f$. If $M$ is screen isoparametric, then each slice $S_t$ is an isoparametric hypersurface of $\{t\}\times F$. Given a point $p\in S_t$, there is a neighborhood $U$ of $p$ in $S_t$ where we have a unit vector field $\eta$ normal to $S_t$ in $\{t\}\times F$. Define
\begin{equation}\label{eq:parametrizacion}
\Phi:(-\varepsilon,\varepsilon)\times U\to M,\qquad \Phi(s,q)=(f(\exp_q(s\eta(q))),\exp_q(s\eta(q))).
\end{equation}

\begin{lema}
If $t\in\mathbb R$ is not a critical value of $f$, then the transformation $\Phi$ is a local diffeomorphism.
\end{lema}

\begin{proof}
We give the proof in the case $\bar c=-1$; the other cases are similar. Hence $\bar M_{\bar c}=-I\times_{\cos t}\mathbb H^{n+1}$. Although for the lemma we only have to calculate $d\Phi_{(0,p)}$, for future purposes we calculate $d\Phi_{(s,q)}:\mathbb R\times T_qS_t\to T_{\Phi(s,q)}M$ at vectors of the form $(v,0)$ and $(0,X)$, where $v\in\mathbb R$ and $X\in T_qS_t$. First note that $\exp_q(s\eta(q))$ is obtained by evaluating at the point $s$ the geodesic departing from $q$ with tangent vector $\eta(q)$. In the hyperboloid model of the hyperbolic space $\mathbb H^{n+1}\subset \mathbb R_1^{n+2}$, we may describe this geodesic as $(\cosh s)q+(\sinh s)\eta(q)$.

With this setting, we have
\[
d\Phi_{(s,q)}(v,0)=\left.\frac{d}{du}(f(\exp_q(((s+uv)\eta(q))),\exp_q((s+uv)\eta(q)))\right\vert_{u=0}
\]
but
\begin{eqnarray*}
\left.\frac{d}{du}\exp_q((s+uv)\eta(q))\right\vert_{u=0} & = & \left.\frac{d}{du}(\cosh(s+uv)q+\sinh(s+uv)\eta(q))\right\vert_{u=0} \\[0.2cm]
 & = & v((\sinh s)q+(\cosh s)\eta(q)).
\end{eqnarray*}

Also, we have
\[
\left.\frac{d}{du}f(\exp_q((s+uv)\eta(q)))\right\vert_{u=0} = v\, df_{\exp_q(s\eta(q))}((\sinh s)q+(\cosh s)\eta(q))
\]
and then
\[
d\Phi_{(s,q)}(v,0)=v(df_{\exp_q(s\eta(q))}((\sinh s)q+(\cosh s)\eta(q)), (\sinh s)q+(\cosh s)\eta(q)),
\]
the first coordinate being different from zero at a non-critical point of $f$ (and $v\ne 0$).

Now we evaluate $d\Phi_{(s,q)}$ at vectors of the form $(0,X)$. Let $\gamma(u)$ be a curve in $S_t$ such that $\gamma(0)=q$ and $\gamma'(0)=X$. Then
\[
d\Phi_{(s,q)}(0,X)=\left.\frac{d}{du}(f(\exp_{\gamma(u)}(s\eta(\gamma(u)))),\exp_{\gamma(u)}(s\eta(\gamma(u))))\right\vert_{u=0};
\]
we have
\begin{eqnarray*}
\left.\frac{d}{du}\exp_{\gamma(u)}(s\eta(\gamma(u)))\right\vert_{u=0} & = & \left.\frac{d}{du}((\cosh s)\gamma(u)+(\sinh s)\eta(\gamma(u)))\right\vert_{u=0} \\
 & = & (\cosh s)X-(\sinh s)A_\eta X;
\end{eqnarray*}
from this we obtain
\[
\left.\frac{d}{du}f(\exp_{\gamma(u)}(s\eta(\gamma(u))))\right\vert_{u=0}=((\cosh s)X-(\sinh s)A_\eta X)(f)=0,
\]
the last equality being consequence that $X$ is tangent to a level set of $f$. In short,
\[
d\Phi_{(s,q)}(0,X)=(0, (\cosh s)X-(\sinh s)A_\eta X ).
\]

Taking an orthonormal basis $E_1,\dots,E_n$ of $T_qS_t$ consisting of eigenvectors of $A_\eta$ with eigenvalues $\nu_i$, we have
\[
d\Phi_{(s,q)}(0,E_i)=(0, ((\cosh s)-\nu_i(\sinh s)) E_i ).
\]

Since these vectors along with $d\Phi_{(s,q)}(v,0)$ ($v\ne 0$) are linearly independent, $d\Phi_{(s,q)}$ is an isomorphism and hence $\Phi$ is a local diffeomorphism. 
\end{proof}

We calculate the coefficients of the metric $g$ of $M$ in terms of the above diffeomorphism $\Phi$. First note that
\[
g(d\Phi_{(s,q)}(v,0),d\Phi_{(s,q)}(v,0))=v^2\langle \eta,\eta\rangle (-\Vert \grad f\Vert^2+(\varrho\circ f)^2)=0.
\]
since we are working in $-I\times_\varrho F$ and $f$ is transnormal relative to $\varrho$. It is clear that
\[
g(d\Phi_{(s,q)}(v,0),d\Phi_{(s,q)}(0,X))=0;
\]
on the other hand, taking the orthonormal eigenvector field $E_i$,
\[
g(d\Phi_{(s,q)}(0,E_i),d\Phi_{(s,q)}(0,E_j))=((\cosh s)-\nu_i(\sinh s)) ((\cosh s)-\nu_j(\sinh s)) \delta_{ij};
\]

The above analysis may be summarized as follows.

\begin{teorema}
Let $(M,g,S^*(TM))$ be a null isoparametric hypersurface of the Lorentzian warped product $-I\times_\varrho F$ with constant curvature $\bar c$, $M$ given as the graph of a transnormal function $f$. Each point $(t,p)\in M$ such that $t$ is a regular value of $f$ has a neighborhood in $M$ diffeomorphic to $(-\varepsilon,\varepsilon)\times S_t$, where $S_t=f^{-1}(t)$.

In particular, let $\lambda_1(t),\dots,\lambda_{l(t)}(t)$ be the distinct screen principal curvatures of $S_t$ with multiplicities $m_1(t),\dots,m_{l(t)}(t)$. If the number $l$ and the multiplicities are constant in a neighborhood of $(t,p)$, $M$ is diffeomorphic to a product $M_0\times M_1\times\cdots\times M_l$, where $\dim M_0=1$ and $\dim M_i=m_i$.

Moreover, the metric in $M$ may be expressed as a degenerate multiply warped product metric given by
\[
ds^2=
\begin{cases}
0\,dt^2+\sum\limits_{i=1}^l\left( \sum\limits_{j=1}^{m_1}((\cosh s)-\nu_i(\sinh s))^2 dx_{ij}^2 \right),&\bar c=-1: \\[0.5cm]
0\,dt^2+\sum\limits_{i=1}^l\left( \sum\limits_{j=1}^{m_1}(1-\nu_i s)^2 dx_{ij}^2 \right), & \bar c=0; \\[0.5cm]
0\,dt^2+\sum\limits_{i=1}^l\left( \sum\limits_{j=1}^{m_1}((\cos s)-\nu_i(\sin s))^2 dx_{ij}^2 \right), & \bar c=1;
\end{cases}
\]
where
\[
\nu_i=\sqrt{2}\lambda_i+\frac{\varrho'}{\varrho}.
\]
\end{teorema}

\section*{Acknowledgements}

The first author is grateful with CONACYT and Facultad de Ciencias, UNAM for the financial support and warm hospitality during the sabbatical year in which this work was developed.

\bibliographystyle{plain}
\bibliography{references3}

\end{document}